\documentclass[11pt,a4paper]{amsart}
\usepackage{graphicx}
\usepackage{comment}
\usepackage{amssymb}
\usepackage[all]{xy}
\usepackage{amsmath}
\usepackage[french,english]{babel}
\usepackage{amssymb,amsmath,amsfonts,amsthm,amscd}
\usepackage{indentfirst,graphicx}
\usepackage[font={scriptsize},captionskip=5pt]{subfig}

\def\nP{{\mathbb{P}^{n-1}}}
\def\P2{{\mathbb{P}^{2}}}
\def\P3{{\mathbb{P}^{3}}}
\def\P5{{\mathbb{P}^{5}}}
\def\P8{{\mathbb{P}^{8}}}

\def\Sh{\hom_s(\mathbb{C}^n,\mathbb{C}^n)}
\def\Shr{\hom_s(\mathbb{C}^{n-r},\mathbb{C}^{n-r})}

\def\ho{\hom(\mathbb{C}^n,\mathbb{C}^n)}
\def\C{{\mathbb{C}}}

\def\R{{\mathbb{R}}}

\def\X{{\mathcal{X}}}
\def\O{{\mathcal{O}}}

\def\R{{\mathcal{R}}}
\def\P{{\mathbb{P}}}

\DeclareMathOperator{\Ima}{Im}
\DeclareMathOperator{\codim}{codim}
\DeclareMathOperator{\rank}{rank}

\DeclareMathOperator{\Projan}{Projan}

\DeclareMathOperator{\colength}{colength}

\newtheorem{theorem}{Theorem}[section]
\newtheorem{lemma}[theorem]{Lemma}     
\newtheorem{corollary}[theorem]{Corollary}

\newtheorem{proposition}[theorem]{Proposition}
\newtheorem{example}[theorem]{Example}
\newtheorem{remark}[theorem]{Remark}

\begin{document}

\title [Symmetric Determinantal Singularities]{Symmetric Determinantal Singularities I: The Multiplicity of the Polar Curve}

\author[T. Gaffney]{Terence Gaffney}\thanks{T.~Gaffney was partially supported by PVE-CNPq Proc. 401565/2014-9}
 \address{T. Gaffney, Department of Mathematics\\
  Northeastern University\\
  Boston, MA 02215}

\author[M. Molino]{Michelle Molino}\thanks{This paper contains work from this author's PhD dissertation at Universidade Federal Fluminense and was partially supported by Coordena\c c\~ao de Aperfei\c coamento de Pessoal de N\' ivel Superior - (Financiamento 001) and National Council for Scientific and Technological Development - CNPq} 
\address {M. Molino, Department of Mathematics\\
  Universidade Federal Fluminense\\
  Niteroi, RJ - Brazil}

\begin{abstract}
This paper is the first part of a two part paper which introduces the study of the Whitney Equisingularity of families of Symmetric determinantal singularities. This study reveals how to use the multiplicity of polar curves associated to a generic deformation of a singularity to control the Whitney equisingularity type of these curves.
\end{abstract}

\maketitle

\selectlanguage{english}

\section*{Introduction}

In this paper and part II \cite{P2}, we study the Whitney equisingularity of families of symmetric determinantal varieties. It is part of a long term effort by several researchers to connect invariants of algebraic objects (rings, ideals and modules) associated with singularities of complex spaces to equisingularity conditions. The project took off with work of Bernard Teissier in the 70s. Teissier, in \cite{CESPCW}, in the case of families of hypersurfaces, with isolated singularities, found integral closure descriptions of equisingularity conditions, Whitney A and B, and controlled these conditions using algebraic invariants, such as multiplicities of ideals. Gaffney, in a series of papers, \cite {MEIG}, \cite{SNHS}, \cite {SIDM}, extended the results of Teissier to families of complete intersection, isolated singularities, hypersurfaces with non-isolated singularities, and then constructed a framework for dealing with isolated singularities in general in \cite{PMIPME}. The approach of \cite{PMIPME} is based on pairs of modules $M,N$, $M\subset N$. The choice of $M$ is canonical: it is the module generated by the partial derivatives of the defining equations of $X$, known as the Jacobian module. The best choice of $N$ was less well understood, as some obvious choices lead to technical difficulties in calculating invariants associated to $N$. Recent work by Gaffney and Rangachev (\cite {Gaff1}) supports the following approach adapted to the symmetric case:

In this first part of the paper, we start defining the symmetric determinantal varieties and some of our objects of study, such as, the tangent space and the Jacobian and Normal modules. We extend some of the results from \cite{Gaff1} to the symmetric case. At the end we give a similiar interpretation, from the rectangular case, of the fiber of the conormal space of the symmetric determinantal variety $S_r$.

In Section two we calculate the multiplicity of the polar curve of $N(\X)$ where $\X$ is a $1$-parameter stabilization of the symmetric determinantal variety $X$. First we analyze the multiplicity for the case where $X=F^{-1}(S_{n-1})$, which is similar to the rectangular case proved by Gaffney and Rangachev in \cite{Gaff1}. After that we move to case where $X=F^{-1}(S_{n-2})$ and prove a formula to calculate the multiplicity based on an intersection number. For this last case, we give a explicit way to calculate the multiplicity by showing an equivalence between $\Projan \R(N)$ and a modification of $X$ based on the presentation matrix of the singularity. This equivalence then gives a decomposition of the multiplicity of the polar of $N$ as a sum of intersection numbers of generic plane sections with the exceptional fiber of the modification. This means we
can compute the  intersection number of the image of the section with the polar variety of $S_r$ of complementary dimension, as a sum of intersection numbers of modules naturally associated with the singularity; these intersection numbers in turn are the colengths of a collection of ideals.
 
 In section three we compute these intersection numbers, called mixed polars, as the alternating sum of intersections of modules which depend only on the presentation matrix, and give an example of a computation for a family of space surfaces. 

\section{Symmetric Determinantal Varieties and Their Properties}

In this section we  prove  results for symmetric determinantal varieties analogous to the general case where the total space is $\ho$ and $\Sigma_r$ is the set of elements of  $\ho$ of rank less than or equal to $r$. We define the normal and Jacobian module of a symmetric determinantal variety and show how to calculate the multiplicity of this pair. At the end, we prove a similar result from \cite{Gaff1} about the fiber of the conormal space of such varieties. 

Let $\Sigma_r$ be the set of all $n\times n$ matrices of rank less than equal to $r$. Then, the set of all $n\times n$ symmetric matrices of rank less than equal to $r$ is given by
$$
S_r=\{A\in \Sh \mid \rank A\le r\}=\Sigma_r\cap \Sh.
$$

Now consider an element of $S_r$ as a map in $\Sh$. Given a map
$$
\begin{array}{lrcc}
F:& \mathbb{C}^q & \longrightarrow & \Sh\\
 & x & \longmapsto & \left(f_{ij}(x)\right)
\end{array}
$$
a determinantal variety $X\in\mathbb{C}^ q$ is the pre-image of $S_r\in\Sh$, with the expected codimension, that is, $X=F^{-1}(S_r)$ and $\codim X=\codim S_r$. 
 if we consider $F$ as the identity map on $\Sh$, then  $S_r$ is trivially a determinantal variety. 
 The following properties are fundamental in this work. The proof for each one of them is similar to the ones presented in \cite{GAC} for the general case. 

\begin{proposition}\label{prop} Let $S_r$ be a symmetric determinantal variety inside of $\Sh$. Then,
\begin{enumerate}
\item $S_r$ is irreducible;
\item The codimension of $S_r$ in the ambient space is $\frac{(n-r)(n-r+1)}{2}$;
\item The singular set of $S_r$ is exactly $S_{r-1}$;
\item The stratification of $S_r$, given by $\left\lbrace S_{i}\backslash S_{i-1}\right\rbrace$, is locally analytically trivial and hence it is a Whitney stratification of $S_r$.
\end{enumerate}
\end{proposition}
\begin{proof}
Let $G(n-r,n)$ be a grassmannian given by all the linear subspaces of dimension $n-r$ in $\C^n$, and $\pi_1$ and $\pi_2$ be the projections 
$$
\xymatrix{
 & \Sh\times G(n-r,n)\ar[dl]_{\pi_1} \ar[dr]^{\pi_2} & \\
\Sh & & G(n-r,n)}
$$
\begin{enumerate}
\item Consider the set
$$
\widetilde{S}_r=\{(A,W)\in\Sh\times G(n-r,n)\mid \left. A\right|_W=0\}.
$$
Projection onto $G(n-r,n)$ exhibits $\widetilde{S}_r$ as an algebraic vector bundle over the grassmannian $G(n-r,n)$. This implies that $\widetilde{S}_r$ is smooth and connected, which means $\widetilde{S}_r$ is irreducible. Clearly $\pi_2$ maps $\widetilde{S}_r$ properly onto $S_r$, showing that $S_r$ is an irreducible variety of $\Sh$.

\item The sets $\widetilde{S}_r$ and $S_r$ are birationally equivalent, which means
$$
\dim S_r=\dim\widetilde{S}_r.
$$    
The dimension of the fiber $\dim(\pi_1^{-1}(W)\cap\widetilde{S}_r)$ is the same for all $W\in G(n-r,n)$. In that case, let us consider the linear space $W$ given by
$$
W=\{(x_1,\ldots,x_n)\in\C^n\mid x_1=x_2=\ldots x_r=0\}.
$$
The set of symmetric matrices $A$ such that $\left. A\right|_W=0$ is given by matrices of the following type
$$
\left(\begin{array}{c|c}
\begin{array}{ccc}
a_{11}&\ldots&a_{1r}\\
\vdots &\ddots & \vdots\\
a_{1r} & \ldots&a_{rr}
\end{array} & 0 \\
\hline
0&0
\end{array}\right).
$$  
This set is isomorphic to $\hom_s(\C^r,\C^r)$ whose dimension is $\frac{r(r+1)}{2}$. Thus, $\widetilde{S}_r$ is a vector bundle of rank $\frac{r(r+1)}{2}$. In that case,
$$
\dim\widetilde{S}_r=\dim G(n-r,n)+\dim(\pi_1^{-1}(W)\cap\widetilde{S}_r)=\frac{2rn-r^2+r}{2}. 
$$ 
Therefore, the codimension of $S_r$ is
$$
\codim S_r=\frac{n(n+1)}{2}-\dim\widetilde{S}_r=\frac{(n-r)(n-r+1)}{2}.
$$
\item First, let $A$ be a point in $S_r\backslash S_{r-1}$, that is, $A$ is a matrix of rank exactly $r$. This means that there is only one point $W\in G(n-r,n)$ such that $\left. A\right|_W=0$. Therefore, $S_{r-1}$ contain the singular set of $S_r$. Now, we need to prove the equality of these sets. On the other hand, if $A\in S_{r-1}$ then the rank of $A$ is less than $r$ and, therefore, there are more than one $W$ satisfying $\left. A\right|_W=0$, meaning that $A$ is a singular point of $S_r$. 
\item Similiar to the proof for the general case seen in \cite{BDDMF}, this can be deduced by induction from the observation that any point $p\in\left\lbrace S_{i}\backslash S_{i-1}\right\rbrace$ has a product of analytic spaces
$$
(S_r,p)\simeq(S'_{r-i},0)\times\left(\C^{\frac{i(2n-i+1)}{2}},0\right) 
$$
where $S'_{r-i}$ is the symmetric determinantal variety inside of the ambient space $\hom_s(\C^{n-i},\C^{n-i})$ and $\frac{i(2n-i+1)}{2}$ is the dimension of $S_i$.  
\end{enumerate}
\end{proof}

One of the ways of studying the equisingularity of families is by using the multiplicity of pairs of modules and their polar curves. For this, we will use two specific modules associated with the symmetric determinantal variety $X$: the Symmetric Normal module, $N(X)$, and the Jacobian module, $JM(X)$. First, we will define the modules $JM(X)$ and $N(X)$, and then we will show some results that will be useful later.

 We first consider the variety $S_r$ in $\Sh$. The \textit{Jacobian module of $S_r$} is a submodule of $\O_{S_r,0}^p$, where $p$ is the number of equations defining $S_r$, generated by the partial derivatives of these set of equations. Whereas the \textit{Symmetric Normal module of $S_r$} is the module given by the first order symmetric infinitesimal deformations of $S_r$. 

As in the general case (\cite{Gaff1}), the symmetric determinantal varieties $S_r$ are \textit{stable}, that is, varieties whose Jacobian module is equal to the module of allowable first order linear infinitesimal deformations, which in this case is the Symmetric Normal module . 

Now, let $X=F^{-1}(S_r)$ be a symmetric determinantal variety, where $F$ is a map from $\C^q$ to $\Sh$, and consider the pullback 
$$
F^*:\O_{S_r,0}^p\longrightarrow\O_{X,0}^p.
$$
Let $\delta_{i,j}$ denote the symmetric matrix with $1$ in the $(i,j)$, $(j,i)$ place, and $0$ elsewhere. The \textit{Jacobian module of $X$} is the module given by the partial derivatives of the equations defining $X$. Now, the \textit{Normal module of $X$} is defined by taking the minors of $A+t\delta_{i,j}$ composing with $F$, taking the derivative with respect to $t$ and then evaluating at $t=0$, which means $N(X)=F^*(N(S_r))$. Since $S_r$ is stable, 
$$
N(X)=F^*(N(S_r))=F^*(JM(S_r)).
$$

In the next section we want to calculate the multiplicity of the polar of $N(\X)$ in some specific cases, so we want to describe $\Projan\R(N(X))$. The description we give will apply equally to $\Projan\R(N(\X))$. 

Let $[T_{i,j}]$ be the $n$ by $n$ matrix of indeterminates, where $T_{i,j}=T_{j,i}$. We can think of $[T_{i,j}]$ as the identity map on $\Sh$. The generators of $N(X)$ are in one to one correspondence with the entries of this matrix by the procedure described above. We also have 

$$
\Projan(\R(N(X)))\simeq \Projan(\O_X[T_{i,j}]/I)
$$
where $I$ is the ideal of relations between the $T_{i,j}$ under the map which sends them to the $i,j$ generator of $\R(N(X))$. We can also see $I$ as the kernel of the map 
$$
\varphi:\O_X[T_{i,j}]\longrightarrow\R(N(X)).
$$
This will help us to prove the next proposition.

\begin{proposition} $I$ contains the entries of the matrix $[T_{i,j}][F]$.
\end{proposition}
\begin{proof} We work first with the variety $S_r$ inside of $\Sh$ and $M$ the identity map with coordinates $m_{ij}$, then extend to the case of general $F$. In fact for this first case it is convenient to work first with  $\Sigma_r$ inside $\ho$, Thus, we have a map 
$$
\begin{array}{lclc}
\varphi: & \O_{\Sigma_r}[T_{i,j}] & \longrightarrow & \R(JM(\Sigma_r))\\
				&	T_{i,j} & \longmapsto & \left(\frac{\partial\Delta_1}{\partial m_{ij}},\dots,\frac{\partial\Delta_q}{\partial m_{i,j}}\right)\\
\end{array}
$$
where $\Delta =(\Delta_1,\ldots,\Delta_q)$ is the vector of minors of $\ho$ of size $r+1$. We first prove that the entries $[T_{i,j}]^t[M]$ and $[M][T_{i,j}]^t$ are in the kernel of $\varphi$.
A typical entry of $[T_{i,j}]^t[M]$ is equal to 
$$
\sum_{t=1}^n T_{ti}m_{ts}
$$
and, therefore,
$$
\varphi\left(\sum_t T_{ti}m_{ts}\right)=\sum_{t=1}^n m_{ts}\varphi(T_{ti})=\sum_{t,q} m_{ts}\frac{\partial\Delta_q}{\partial m_{ti}}.
$$  
Now, let us fix $q$. Observe if $\Delta_q$ is the minor of a submatrix which does not contain $m_{ti}$ then 
$$
\frac{\partial\Delta_q}{\partial m_{ti}}=0.
$$
So we may assume $\Delta_q$ contains $m_{ti}$. Now, $\frac{\partial\Delta_q}{\partial m_{ti}}=0$ is just the cofactor of $m_{ti}$ so 
$$
\sum_{t,q} m_{ts}\frac{\partial\Delta_q}{\partial m_{ti}}
$$
is just the expansion of $\Delta_q$ but with $m_{ts}$ replacing $m_{ti}$. So this is either zero if $s\neq i$ and $m_{ts}$ is already a part of $\Delta_q$, or $\Delta_q$ if $s=i$, or another minor if $s\neq i$ and $m_{ts}$ is not part of $\Delta_q$. In any event, all terms with fixed  $q$ are zero. The computation for $[M][T_{i,j}]^t$ is similar.

Now we pass to the symmetric case. We view $S_r$ as a subset of $\ho$. Its equations are $\{\Delta_q\}$ as before and $\{(m_{ji}-m_{ji})\}$. So we have the map
$$
\begin{array}{lclc}
\varphi: & \O_{S_r}[T_{i,j}] & \longrightarrow & \R(JM(S_r))\\
				&	T_{i,j} & \longmapsto & \left(\frac{\partial\Delta}{\partial m_{ij}},\frac{\partial(m_{st}-m_{ts})}{\partial m_{i,j}}\right)\\
\end{array}
$$
For this case, we claim that the entries of $([T]^t+[T])[M]$ are in the kernel of $\varphi$. Indeed, the condition for a element to be in the kernel of $\varphi$ falls into $2$ parts: the first is $v\cdot\Delta=0$ and the second is that $v\cdot(m_{ij}-m_{ji})=0$. The computations for general matrices done above, apply to satisfying the first part of the condition. 
Note that if the entries of $[M][T_{i,j}]^t$ satisfy the first part, then so do the entries of $\left([M][T_{i,j}]^t\right)^t=[T_{i,j}][M]$ since $[M]$ is symmetric. So the entries of $([T]^t+[T])[M]$ satisfy the first part. 

We claim the sum satisfies the second part. The element of $([T]^t[M]+[T][M]$ in position $ij$ is
$$
\sum_t T_{ti}m_{ti}+\sum_t T_{it}m_{ti}.
$$
Under $\varphi$ this produces the vector field
$$
\sum m_{ti}\frac{\partial}{\partial m_{ti}}+\sum m_{ti}\frac{\partial}{\partial m_{it}}.
$$
Applying this to $(m_{ti}-m_{it})$ gives us
$$
m_{ti}\frac{\partial m_{ti}}{\partial m_{ti}}-m_{ti}\frac{\partial m_{it}}{\partial m_{it}}=0.
$$
If we restrict the matrix to the symmetric case, that is, $[T_{i,j}]^t=[T_{i,j}]$ then the entries of the matrix $[T_{i,j}][M]$ are zero as are the entries of $[M][T_{i,j}]$.

Passing to the general case, $F$ induces maps, $F^*\: \O_{\Sigma_r}[T_{i,j}]\to \O_{X}[T_{i,j}]$ and $\hat F^*\: \mathcal R(JM(\Sigma_r))\to \mathcal R(F^*(JM(\Sigma_r)))$.
The commutativity of the induced diagram then implies that the entries of the matrix $[T_{i,j}][F]$ are in the kernel of the map from $\O_{X}[T_{i,j}]$ to $\mathcal R(F^*(JM(\Sigma_r)))$, which implies the result.
\end{proof}

\begin{remark} The proposition above allows us to conclude that $[T_{i,j}]F=0$ are some of the equations of $\Projan(\R(N(X)))$.
\end{remark} 

In the next results we describe the fiber of the conormal modification of $S_r$. 
The polar varieties of $S_r$ are intimately connected with our invariant. Since the polar varieties are obtained by intersecting the conormal $C(S_r)$ of $S_r$ with enough generic hyperplanes, then projecting to $S_r$, whether the polar varieties are empty or not depends on the dimension of the fibers  $C(S_r)$.

Let $C=\{(A,B), A\in S_r, B\in {\mathbb{P}}S_{n-r}, AB=0\}$. Note that $C$ is the set of points in $S_r\times{\mathbb{P}}S_{n-r}$ which satisfy the equations of the previous proposition. Our next goal is to show that $C=C(S_r)$.

\begin{lemma} $C_{|S_r-S_{r-1}}=C(S_r)_{|S_r-S_{r-1}}$
\end{lemma}
\begin{proof} If $(A,B)\in C(S_r)_{|S_r-S_{r-1}}$, then $AB=0$ by the previous proposition hence in $C_{|S_r-S_{r-1}}$. So suppose  $(A,B)\in C_{|S_r-S_{r-1}}$. Since $A$and $B$ are symmetric $AB=0$ implies $AB^t=B^tA=0$.This implies by \cite{Gaff1}, that $B$ defines a tangent hyperplane to $\Sigma^r-\Sigma^{r-1}\subset \ho$ at $A$.
Since $B$ is symmetric and $S_r-S_{r-1}\subset \Sigma^r-\Sigma^{r-1}$ is a smooth embedding, $B$ defines a tangent hyperplane to $S_r-S_{r-1}$ at $A$ in $\Sh$.

\end{proof}

\begin{corollary} \label {open case}$\overline{C_{|S_r-S_{r-1}}}=C(S_r)$, and $C(S_r)$ is an component of $C$
\end{corollary}
\begin{proof} Since  $C_{|S_r-S_{r-1}}=C(S_r)_{|S_r-S_{r-1}}$,  $\overline {C_{|S_r-S_{r-1}}}=\overline{C(S_r)_{|S_r-S_{r-1}}}$. But $\overline{C(S_r)_{|S_r-S_{r-1}}}=C(S_r)$. Since $S_r$ is irreducible, so is $C(S_r)$, so it is a component of $C$. 
\end{proof}

\begin{remark} \label{group action} There is a $GL(n, \C)$ action on $C$ which preserves the rank of $A$ and $B$, given by $S\cdot(A,B)=(S^tAS, S^{-1}B(S^{-1})^t)$.\end{remark}

For the next proof, it is convenient to decompose the ${\mathbb C}^n$ as ${\mathbb C}^s\oplus{\mathbb C}^{r-s}\oplus{\mathbb C}^{n-r}$, with $s<r<n$ and to write 
${\mathbb C}^s\oplus{\mathbb C}^{r-s}$ as ${\mathbb C}^r$, ${\mathbb C}^{r-s}\oplus{\mathbb C}^{n-r}$ as ${\mathbb C}^{n-s}$. Let $I_u$ denote the identity map on ${\mathbb C}^u$, with $u=s, r, n-r, n-s, r-s$

\begin{lemma} Suppose $A=I_s\in S_r$, $s<r<n$, then $C_A=C(S_r)_A$.
\end{lemma}
\begin{proof} $\hom_s(\mathbb{C}^{n-s},\mathbb{C}^{n-s})$ is embedded in $\Sh$ by the inclusion of ${\mathbb C}^{n-s}$ in ${\mathbb C}^{n}$ and by extension over 
${\mathbb C}^{s}$ by $0$. With this identification the fiber, $C_A$, of $C$ over $A$ is ${\mathbb P}(S_r)$, $S_r\subset \hom_s(\mathbb{C}^{n-s},\mathbb{C}^{n-s})$.
If $B\in {\mathbb P}(S_r)$, it suffices to prove $B\in C(S_r)_A$ with rank of $B=n-r$. Now we can choose an invertible $S$ such that $S^tBS=I_{n-r}$. Now
$$0=I_r I_{n-r}=I_rS^tBS=I_rS^t(B).$$
Consider the line $L$ in $S_r$ parameterized by $I_s+t(SI_{r-s}S^t)$. Since $(I_s+t(SI_{r-s}S^t))B=I_sB+t(S(I_{r-s}S^t)B)=0$, $(I_r+t(SI_{n-r}S^t),B)$ is a line in $C$ passing through $A$ at $t=0$. Since, for $t\ne 0$ the points on $L$ have rank $r$, by the previous lemma, the line in $C$ lies in $C(S_r)$, hence $B\in C(S_r)$.
 \end{proof}
\begin{theorem} $C(S_r)_A=\{(A,B), A\in S_r, B\in {\mathbb{P}}S_{n-r}, AB=0\}$.
\end{theorem}
\begin{proof} By \ref{open case}, $C(S_r)_A\subset \{(A,B), A\in S_r, B\in {\mathbb{P}}S_{n-r}, AB=0\}$. So,  suppose $B \in {\mathbb{P}}S_{n-r}$ $AB=0$. Using the group action of \ref{group action} we can move $(A,B)$ to $(I_s,\tilde B)$ in $C$, $s$ the rank of $A$. Then there exists a curve by the previous lemma which lies in $C(S_r)$ passing through $(I_s, \tilde B)$. By the group action, we can move $(I_s, \tilde B)$ back to $(A,B)$ and the curve along with it. 
\end{proof}

We would like to describe the fiber $C(S_r)_A$ in terms of linear spaces associated with $A$.

 Given $M\in \Sh$, $C(M)$ denotes the vectors in ${{\mathbb C}^{n}}^*$ which annihilate the image of $M$, and $K^*(M)$ the quotient of ${{\mathbb C}^n}^*$ by those vectors which annihilate the kernel of $M$. Note that every element of $hom(K^*(M),C(M))$ has a well defined extension to 
 $hom({{\mathbb C}^n}^*,{{\mathbb C}^n}^*)$, by the inclusion of $C(M)$ in $\mathbb C^{n*})$, and the extension by $0$ over those vectors which annihilate the kernel of $M$, so we can view $hom(K^*(M),C(M))$ as a subspace of $hom({{\mathbb C}^n}^*,{{\mathbb C}^n}^*)$. Denote this embedding by $\Phi$. We define $hom_s(K^*(M),C(M))$ to be 
 $$\Phi^{-1}(\Phi(hom(K^*(M),C(M)))\cap \Sh).$$
 
Note that since $M$ is symmetric, then $v\in C(M)$ if and only if $v^*\in {mathbb C}^n$ is in $K(M)$. 
 Let $X_r(M)$ be elements of ${\mathbb P}(hom_s(K^*(M),C(M)))$ of rank less than or equal to $r$.
 
 \begin{theorem} \label{fiber} Suppose $M\in S_r$. Then 
  $C(S_r)_M\approx X_{n-r}(M)$. If $K(M)\cap M({\mathbb C}^n)=0$ then
 \begin{enumerate}\item $C(S_r)_M\simeq {\mathbb P}S_{n-r}\subset {\mathbb P}hom_s(C(M), C(M))$.
 \item $C(S_r)_M\simeq  {\mathbb P}S_{n-r}\subset {\mathbb P}hom_s(K(M), K(M))$.
 \item $C(S_r)_M\simeq  {\mathbb P}S_{n-r}\subset {\mathbb P}hom_s(K^*(M), K^*(M))$.
 \end{enumerate}

 \end{theorem}

\begin{proof} Suppose $B\in hom_s(K^*(M),C(M))$. Since $\Phi(B)$ is  symmetric, the element of $\Sh$ induced by duality has the same matrix as  $\Phi(B)$, so we can consider $M(\Phi(B))$ and this is the zero element as the image of $\Phi(B)$ lies in $C(M)$. Since $B\in X_{n-r}(M)$, the rank of $\Phi(B)\le n-r$, so $Phi(B)\in C(S_r)_M$.

Suppose $B\in\Sh$ and $MB=0$. Then the row space and the column space of $B$ lie in $C(M)$. Further, 
$$ann(K(M)=im M^*\subset K(B).$$ 
So, since $B$ is $0$ on $ann(K(M))$, $B=\Phi(\tilde B)$ for some $\tilde B$. Since $\Phi$ is an embedding, this finishes the first part of the proof.

Now suppose $K(M)\cap M({\mathbb C}^n)=0$. Suppose $v\in K(M)$, $v\ne 0$. Consider $v^*\in {{\mathbb C}^n}^*$. We claim $v^*\notin ann(K(M))$. If it were, then $v^*\in im(M^*)$ which implies $v\in im(M)$. This implies that the canonical map $\phi: K(M)\to K^*(M)$ is an embedding. Since these spaces have the same dimension and $\phi$ is linear, $\phi$ is an isomorphism. Since $C(M)\simeq K(M)$ by the symmetry of $M$ and duality the result follows.\end{proof}

\section{Multiplicity of the Polar Varieties of $N(\X)$}\label{MPV}

In this section we describe the polar variety of a symmetric determinantal variety and show a formula to calculate the multiplicity of a polar curve of $N$ in a deformation to a stabilization. To be more precise, we calculate the multiplicity of the polar of $N(\X)$ when $\X$ is the symmetric determinantal variety given by $\X=\widetilde{F}^{-1}(S_{n-1})$ or $\X=\widetilde{F}^{-1}(S_{n-2})$, which is reduced to an intersection number. At the end, for the case where $X=F^{-1}(S_{n-2})$, as in \cite{Gaff1}, we give an explicit way to calculate these intersection numbers as an alternate sum of colength of ideals.

The \textit{polar variety of codimension $l$} of $S_r$, denoted by $\Gamma_l(S_r)$, at the origin, is the germ given by intersecting $C(S_r)\cap (S_r\times H_{h-1-d})$ with $l$ hyperplanes, then projecting to $S_r$. Altogether, we need to intersect $C(S_r)$ with $h-1-d+l$ hyperplanes and then project to $S_r$ as below
$$
\left.\begin{array}{c}
\pi:C(S_r)\cap (S_r\times H_{h-1-d+l})\\
\downarrow\\
S_r
\end{array}\right.
$$ 

Sometimes it is better to work with the dimension of the polar variety instead of the codimension. In these cases, if we want the polar variety $\Gamma^l(S_r)$ of dimension $l$ we are going to intersect $C(S_r)$ with $h-1-l$ hyperplanes and then project to $S_r$. That means
$$
\Gamma^l(S_r)=\pi\left(C(S_r)\cap (S_r\times H_{h-1-l}\right).
$$

\begin{proposition}\label{empty}
The polar variety $\Gamma^l(S_r)$ of dimension $l$ is empty for all
$$
l\leq\frac{r(r+1)}{2}-1.
$$
\end{proposition}
\begin{proof}
Let $c$ denote the codimension of $S_{n-r}$ in $\Sh$. By taking $M=0$ in the theorem \ref{fiber} we have
$$
C(S_r)_0=X_{n-r}(0)\simeq X_{n-r}.
$$
Thus,
$$
\dim C(S_r)_0=\dim X_{n-r}=\dim S_{n-r}-1=(h-c)-1=(h-1)-c.
$$
Since $C(S_r)_0\subset\mathbb{P}(\Sh)$, $\codim C(S_r)_0=c$. Let us consider the map  
$$
\pi:C(S_r)\cap (S_r\times H_{j})\longrightarrow S_r.
$$
Using $j$ hyperplanes we have the polar variety of dimension $h-1-j$. Now, the dimension of the fiber of $\pi$ over $0$ is given as follows:
$$
\dim \pi^{-1}(0)=\dim p^{-1}(0)-j=\dim \left(C(S_r)_0\right)-j=(h-c)-j-1.
$$
If $j\geq h-c$, then $\pi^{-1}(0)=\emptyset$, which means that the polar variety $\Gamma^{h-1-j}(S_r)$ is empty. Therefore, all the polar varieties of dimension less than equal to
$$
(h-1)-(h-c)=c-1=\frac{(n-(n-r))(n-(n-r)+1)}{2}-1=\frac{r(r+1)}{2}-1
$$
are empty.
\end{proof}

Our focus now is to calculate the multiplicity of the polar of $N(\X)$, where $\X=\widetilde{F}^{-1}(S_{n-1})$. For that we need to describe the $\Projan\mathcal{R}(N(X))$. The description we give will apply equally to $\Projan\mathcal{R}(N(\X))$.

Consider the map
$$
\begin{array}{lclc}
F: & \mathbb{C}^q & \longrightarrow & \Sh\\
					& x & \longmapsto & \left(f_{ij}(x)\right)
\end{array}
$$
whose entries are complex analytic functions with $f_{ij}(x)=f_{ji}(x)$ and $X$ has expected codimension. In this case, $X$ is a hypersurface of dimension $q-1$. As we have seen, the normal module $N(X)$ is an ideal generated by $\frac{n(n+1)}{2}$ polynomials, and, for each $x\in X$, its row space is generated by one vector that we call $v_x$. The projective analytic spectrum of the Rees Algebra is given by, 
$$
\Projan\mathcal{R}(N(X))=\overline{\left\lbrace(x,l) | x\in X_{reg} \text{ and } l\in\mathbb{P}(\left\langle v_x\right\rangle)\right\rbrace}\subseteq X\times\mathbb{P}^\frac{n(n+1)}{2}.
$$
If $x$ is a smooth point of $X$, then $F(x)$ has rank $n-1$. By one of the equations of $\Projan\mathcal{R}(N(X))$ we have
$$
F(x)(l_{ij})=0
$$
which means that all the columns of $(l_{ij})$ (consequently all the rows) are in $\ker F(x)$. Since $\rank F(x)=n-1$, the dimension of $\ker F(x)$ is equal to 1; therefore, each column of $[T_{ij}]$ is a multiple of a fixed vector in $\ker F(x)$, which implies that $[T_{ij}]$ is a matrix of rank $1$. Now, define the set $X_F$, contained in $X\times\mathbb{P}^{n-1}\times\mathbb{P}^{n-1}$, by
$$
X_F=\overline{\left\lbrace(x,l_1,l_2) | x\in X_{reg} \text{ and } (l_1,l_2)\in\Delta\left(\mathbb{P}(\ker (F(x))\times\mathbb{P}(\ker (F(x))\right)\right\rbrace}.
$$ 
where $\Delta$ is the diagonal.

\begin{proposition}\label{simeq} $X_F\simeq \Projan\mathcal{R}(N(X))$ as sets.
\end{proposition}
\begin{proof}
Both sets are defined by the closures of points over the smooth set of $X$. So, let us work on this set. Consider the Segre embedding 
$$
\begin{array}{lclc}
\varphi: & \mathbb{P}^{n-1}\times\mathbb{P}^{n-1} & \longrightarrow & \mathbb{P}^{n^2-1}\\
			&		((S_1,\ldots,S_n),(T_1,\ldots,T_n)) & \longmapsto & \left(\begin{array}{ccc}
												S_1T_1 & \cdots & S_1T_n\\
												\vdots & \ddots & \vdots\\												
												S_nT_1& \cdots & S_nT_n
												\end{array}\right)
\end{array}
$$

We need to prove that $\left(\Projan\mathcal{R}(N(X))\right)_x\subseteq\varphi\left((X_F)_x\right)$, where $x$ is a smooth point of $X$. For this, let $l=\left(l_{ij}\right)$, with $l_{ij}=l_{ji}$, be a point in $\left(\Projan\mathcal{R}(N(X))\right)_x$. By the properties of $\Projan\mathcal{R}(N(X))$, for all $i$ we have $(l_{i1},l_{i2},\ldots,l_{in})\in\ker F(x)$.  Since $\varphi\left(\P^{n-1}\times\P^{n-1}\right)$ is the set of all matrices in $\P^{n^2-1}$ whose rank is equal to one, $l\in\varphi\left(\P^{n-1}\times\P^{n-1}\right)$ and, therefore, there exists $(s,t)\in\P^{n-1}\times\P^{n-1}$ such that $\varphi(s,t)=l$.

Now, we need to show that $(s,t)\in\Delta\left(\P(\ker (F(x))\times\P(\ker (F(x))\right)$. The Segre embedding gives us
$$
\begin{array}{c}
\left(\begin{matrix}
		l_{11}&l_{12}&\ldots&l_{1n}\\
		\vdots&\vdots&\ddots&\vdots\\
		l_{1n}&l_{2n}&\ldots&l_{nn}
\end{matrix}\right)=\left(\begin{matrix}
		s_1t_1&\ldots&s_1t_n\\
		\vdots&\ddots &\vdots\\
		s_nt_1&\ldots&s_nt_n
\end{matrix}\right)\\
\Downarrow\\
s_i(t_1,\ldots,t_n)=t_i(s_1,\ldots.s_n)=(l_{i1},\ldots,l_{in}).
\end{array}
$$
This implies that $s=t$ in $\P(\ker F(x))$ and, therefore, $l\in\varphi((X_F)_x)$.

As we have seen, $\dim\ker F(x)=1$, and dimension of $(\Projan\mathcal{R}(N(X)))_x$ is one less than the rank of the Jacobian module of $X$, which is the expected codimension of $X$. Thus, $\dim (X_F)_x=\dim (\Projan\mathcal{R}(N(X)))_x=0$. Since $\varphi\left((X_F)_x\right)$ is irreducible, $(\Projan\mathcal{R}(N(X)))_x$ is closed and both sets have the same dimension, $\left(\Projan\mathcal{R}(N(X))\right)_x = \varphi\left((X_F)_x\right)$. 
\end{proof}

We will use this result to compute the degree over the base $\mathbb{C}$ of the polar variety of dimension $1$ of $N(\mathcal{X})$, that is $\deg_{\C}\Gamma_{q-1}(N(\X))$, where $\mathcal{X}$ is the total space of the deformation, and a generic fiber is smooth. Consider the projection map 
$$
p_1:\X\times\mathbb{P}^{n-1}\times\mathbb{P}^{n-1}\longrightarrow\mathbb{P}^{n-1}
$$
and let $h_1$ be the pullback of a hyperplane class of $\mathbb{P}^{n-1}$ via the projection map, and $h$ be a hyperplane class of the diagonal. Denote the fiber over the origin in $\X$ of $\Projan\R(N(\X))$ by $E$, and consider $q-1$ the dimension of $X$. The degree of $\Gamma_{q-1}(N(\X))$ over $\C$ is calculated as the next theorem shows.

\begin{theorem}\label{HC} Suppose $\X$ is a stabilization of $X$, with base $\C$. The degree of the polar curve $\Gamma_{q-1}(N(\X))$ over $\C$ at the origin is
$$
\deg_{\C}\Gamma_{q-1}(N(\X))=(2h_1)^{q-1}\cdot\varphi^*E.
$$
\end{theorem}
\begin{proof} By definition, the degree of $\Gamma_{q-1}(N(\X))$ over $\C$ at the origin is the degree of the projection to $\C$ at the origin of $\Gamma_{q-1}(N(\X))$. The generic rank of $N(\X)$ is one, since it is an ideal, and $\Projan\R(N(\X))$ has dimension $q$, with generic fiber dimension $0$. The fiber of the exceptional divisor $E$ has dimension one less than $\Projan\R(N(\X))$, that is, $q-1$. The polar curve $\Gamma_{q-1}(N(\X))$ is given by intersecting $\Projan\R(N(\X))$ with ${q-1}$ generic hyperplanes of $\P^\frac{n(n+1)}{2}$ and projecting to $\X$ by $p$. The degree is calculated as follows: the intersection $\Projan\R(N(\X))\cap H^{q-1}$ is a curve in its ambient space, so the degree of the projection of $\Projan\R(N(\X))\cap H^{q-1}$ over $\C$ is well-defined. By conservation of number, the degree of the polar variety  $\Gamma_{q-1}N(\X)$ is the same as the degree
of $E$ as a projective scheme embedded in $\P^{\frac{n(n+1)}{2}}$. Thus, $\deg_{\C}\Gamma_{q-1}(N(\X))=|E\cap h^{q-1}|$. Now, consider the diagram 
$$
\xymatrix{
\X_F\cap\varphi^{-1}(h^{q-1}) \ar[r] \ar[dr] & \Projan\R(N(\X))\cap h^{q-1} \ar[d]\\
 & \C }
$$ 
By proposition \ref{simeq}, the map from $\X_F\cap\varphi^{-1}(h^{q-1})$ to $\Projan\R(N(\X))\cap h^{q-1}$ is an isomorphism. Moreover, $|E\cap h^{q-1}|$ has the same numbers of intersection points as $|\varphi^{-1}(E)\cap\varphi^{-1}(h^{q-1})|$. Thus,
$$
\begin{array}{ccl}
|\varphi^{-1}(E)\cap\varphi^{-1}(h^{q-1})| & = &\varphi^*(E)\cdot\varphi^*(h^{q-1})\\
 & = & \varphi^*(E)\cdot(2h_1)^{q-1}.
\end{array}
$$
Therefore, 
$$
\deg_{\C}\Gamma_{q-1}(N(\X))=(2h_1)^{q-1}\cdot\varphi^*E.
$$
\end{proof}

Now, we are interested in knowing what happens if $X$ has codimension $3$. For this case, let $X_F$ be defined as follows:
$$
X_F=\overline{\left\lbrace(x,l_1,l_2) | x\in X_{reg} \text{ and } (l_1,l_2)\in\mathbb{P}(\ker (F(x))\right\rbrace}\subseteq X\times\mathbb{P}^{n-1}\times\mathbb{P}^{n-1}.
$$

Take $(x,l)\in\Projan\mathcal{R}(N(X))$, where $l=(l_{ij})$, with $l_{ij}=l_{ij}$. Since $\rank F(x)=n-2$, we have $\dim \ker F(x)=2$, which means that the rank of $l$ has to be $1$ or $2$. For this case, we will not have an isomorphism as before. However, we are still able to exhibit a specific map that will be useful to calculate the degree of $\Gamma (N(\mathcal{X}))$ over $\mathbb{C}$, as we will show in the next results. For that, consider the group action given by:
$$
\begin{array}{lclc}
G: & Gl(n)\times \Sh & \longrightarrow & \Sh\\
 & (A,M) & \longmapsto & A^tMA
\end{array}
$$ 

\begin{lemma}\label{rank12} Let $G$ be the group action defined above, and consider the map
$$
\begin{array}{lclc}
\Phi: & \Sigma_1 & \longrightarrow & S_2\\
		& A & \longmapsto & A + A^t
\end{array}
$$
Then, $\Phi$ is equivariant over the action of $G$, that is, 
$$
\Phi(M\cdot A)=M\cdot\Phi(A)
$$
for all $M\in Gl(n)$ and $A\in\Sigma_1$.
\end{lemma}
\begin{proof} Let $M$ be an element in $Gl(n)$ and $A\in\Sigma_1$. Thus,
$$
\begin{array}{rcl}
\Phi(M\cdot A)&=&\Phi(M^tAM)\\
            &=&M^tAM + (M^tAM)^t\\
            &=&M^tAM + M^tA^tM\\
            &=&M^t(AM + A^tM)\\
            &=&M^t(A + A^t)M\\
            &=&M^t\Phi(A)M\\
            &=&M\cdot\Phi(A).
\end{array}
$$
\end{proof}

\begin{corollary} $\Phi$ carries orbits of $G$ in $\Sigma_1$ to orbits in $S_2$.
\end{corollary}
\begin{proof} Follows directly from lemma \ref{rank12}.
\end{proof}

\begin{proposition}\label{diffeo} $\Phi$ is a $2-1$ covering map, whose critical set is $S_1$, and 
$$
\Phi:\left(\Sigma_1\backslash S_1\right)\longrightarrow\left(S_2\backslash S_1\right)
$$
is a local diffeomorphism.
\end{proposition} 
\begin{proof} Let us start by showing that $\Phi$ is a $2-1$ branched cover. First, it is clear that $\Phi$ is at least $2-1$ cover since for all matrices $A$ in $\Sigma_1$ we have $\Phi(A)=\Phi(A^t)$. Now, suppose $A\in S_2$ has rank $2$; then there are only two preimages for $A$. Indeed, by the group action, we can take $A$ as the matrix
$$
A=\left(\begin{array}{ccccc}
			0 & 1 & 0 & \cdots & 0\\
			1 & 0 & 0 & \cdots & 0\\	
			0 & 0 & 0 & \cdots & 0\\
			\vdots & \vdots & \vdots & \ddots & \vdots\\									
			0 & 0 & 0 & \cdots & 0
			\end{array}\right).
$$
If $B$ is a preimage of $A$, then $B+B^t=A$. Let us write $B$ as follows:
$$
B=\left(\begin{array}{c|c}
\begin{matrix}
b_{11}&b_{12}\\
b_{21}&b_{22}
\end{matrix} & C \\
\hline
D&B_1
\end{array}\right).
$$  
The first thing we can say is that $b_{ii}=0$ and $b_{ij}+b_{ji}=0$, except for $b_{21}+b_{12}=1$. Therefore, $B_1$ is a skew symmetric matrix of rank less than equal to $1$ whose diagonal elements are $0$. However, the only skew symmetric matrix of rank less than or equal to $1$ is the zero matrix, so we claim that $B_1$ is the such matrix. Indeed, if some $b_{ij}$ is not equal to $0$, then $b_{ji}\neq 0$, which means that there exists a $2\times 2$ minor with these elements on it that has a non-zero determinant. But this is impossible since $B$ has rank less than 2. Thus, the preimage of $A$ is contained in the set
$$
\left\lbrace B\in\Sigma_1\quad\vline\quad B=\left(\begin{array}{c|c}
\begin{matrix}
0&b_{12}\\
b_{21}&0
\end{matrix} & C \\
\hline
D&0
\end{array}\right)\right\rbrace.
$$  
Now, let us analyze $b_{12}$ and $b_{21}$. By the equation  $B+B^t=A$ and the fact that $B$ has rank less than $2$, we have
$$
\left\lbrace\begin{array}{ccc}
b_{12}b_{21}&=&0\\
b_{12}+b_{21}&=&1
\end{array}\right..
$$ 
This means we have only two options for the upper left corner of $B$, 
$$
\left[\begin{matrix}
0&0\\
1&0
\end{matrix}\right] \text{  or  } 
\left[\begin{matrix}
0&1\\
0&0
\end{matrix}\right].
$$  
The last step of this part is to analyze $C$ and $D$. If $b_{21}=1$, then the second column of $D$ and the first row of $C$ are zero; and by $B+B^t=A$, the second column of $D$ being $0$ imples that the second row of $C$ is also zero, and the first row of $C$ being $0$ implies that the first column of $D$ is zero. By the same analogy, with $b_{12}=1$, we also have $C$ and $D$ as zero matrices. Therefore, the only preimages of $A$ are
$$
\left(\begin{array}{ccccc}
			0 & 1 & 0 & \cdots & 0\\
			0 & 0 & 0 & \cdots & 0\\	
			0 & 0 & 0 & \cdots & 0\\
			\vdots & \vdots & \vdots & \ddots & \vdots\\									
			0 & 0 & 0 & \cdots & 0
			\end{array}\right) \text{  or  }
\left(\begin{array}{ccccc}
			0 & 0 & 0 & \cdots & 0\\
			1 & 0 & 0 & \cdots & 0\\	
			0 & 0 & 0 & \cdots & 0\\
			\vdots & \vdots & \vdots & \ddots & \vdots\\									
			0 & 0 & 0 & \cdots & 0
			\end{array}\right).
$$   
Note that the above argument shows that $\Phi:\left(\Sigma_1\backslash S_1\right)\longrightarrow\left(S_2\backslash S_1\right)
$ is surjective and since all symmetric matrices of rank $1$ have only one pre-image, $S_1$ is the singular set. Now, we need to show that $p\left(\Sigma_1\backslash S_1\right)$ is a local diffeomorphism. For this, let $B$ be the matrix of rank $1$ given by
$$
B=\left(\begin{array}{ccccc}
			0 & 0 & 0 & \cdots & 0\\
			1 & 0 & 0 & \cdots & 0\\	
			0 & 0 & 0 & \cdots & 0\\
			\vdots & \vdots & \vdots & \ddots & \vdots\\									
			0 & 0 & 0 & \cdots & 0
			\end{array}\right).
$$
By the group action, in order to prove the local diffeomorphism it is enough to show that $D_B\Phi$ has maximal rank. Since $\Phi$ is linear, we have
$$
\begin{array}{lclc}
D_B\Phi: & T_B\Sigma_1 & \longrightarrow & T_{\Phi(B)}S_2\\
		& A & \longmapsto & A + A^t
\end{array}
$$
Now, let us calculate these tangent spaces. As we know, the tangent space of a determinantal variety at a smooth point is equal to 
$$
T_B\Sigma_1=\{C\in \ho\mid C(\ker B)\subset \Ima B\}.
$$
In this case, $\ker B$ is generated by the vectors 
$$
\ker B=\left\langle (0,1,0,\ldots,0),(0,0,1,0,\ldots,0),\ldots,(0,\ldots,0,1)\right\rangle
$$
and $\Ima B=\{(0,t,0,\ldots,0)\mid t\in\C\}$. Thus, the matrix $C$ applied to the kernel vectors gives us
$$
C(0,\ldots,0,\underbrace{1}_{\textit{position } i},0,\ldots,0)=\left(\begin{matrix}
c_{1i}\\
c_{2i}\\
\vdots\\
c_{ni}
\end{matrix}\right)
$$
which means $c_{1i}=c_{3i}=\ldots=c_{ni}=0$ and $c_{2i}$ is any complex number, for all $i=2,\dots,n$. Therefore, if $C$ is a matrix in $T_B\Sigma_1$, then 
$$
C=\left(\begin{matrix}
			c_{11} & 0 & \cdots & 0\\
			c_{21} & c_{22} & \cdots & c_{2n}\\	
			c_{31} & 0 & \cdots & 0\\
			\vdots & \vdots & \ddots & \vdots\\									
			c_{n1} & 0 & \cdots & 0
			\end{matrix}\right).
$$
For the tangent space of $S_2$ at a smooth point we have that
$$
T_{\Phi(B)}S_2=\{D\in \Sh\mid D(\ker (\Phi(B)))\subset \Ima (\Phi(B))\}
$$
where $\Phi(B)=A$. In that case,
$$
\begin{array}{rcl}
\ker A&=&\left\langle (0,0,1,0,\ldots,0),(0,0,0,1,0,\ldots,0),\ldots,(0,\ldots,0,1)\right\rangle\\
\Ima A&=&\{(s,t,0,\ldots,0) \mid s,t\in\C\}.
\end{array}
$$
Thus, the matrix $D$ applied to the kernel vectors gives us
$$
D(0,\ldots,0,\underbrace{1}_{\textit{position } i},0,\ldots,0)=\left(\begin{matrix}
d_{1i}\\
d_{2i}\\
\vdots\\
d_{ni}
\end{matrix}\right).
$$
which means $d_{3i}=d_{4i}=\ldots=d_{ni}=0$ and $d_{1i},d_{2i}$ are any complex number, for all $i=3,\dots,n$. Therefore, if $D$ is a matrix in $T_AS_2$, then 
$$
D=\left(\begin{matrix}
			d_{11} & d_{12}& d_{13} & \cdots & d_{1n}\\
			d_{12} & d_{22} & d_{23} & \cdots & d_{2n}\\	
			d_{13} & d_{23} & 0 & \cdots & 0\\
			\vdots & \vdots & \vdots & \ddots & \vdots\\								
			d_{1n} & d_{2n} & 0 & \cdots & 0
			\end{matrix}\right).
$$
Finally, let us prove that $D_B\Phi$ is surjective. If $D$ is a symmetric matrix in $T_AS_2$, take $C$ as follows
 $$
C=\left(\begin{matrix}
			\frac{1}{2}d_{11} & 0 & \cdots & 0\\
			d_{12} & \frac{1}{2}d_{22} & \cdots & d_{2n}\\	
			d_{13} & 0 & \cdots & 0\\
			\vdots & \vdots & \ddots & \vdots\\									
			d_{1n} & 0 & \cdots & 0
			\end{matrix}\right).
$$
Thus,
$$
D_B\Phi(C+C^t)=D.			
$$
\end{proof}

\begin{corollary}\label{surjective} $\Phi:\Sigma_1\longrightarrow S_2$ is surjective.
\end{corollary}
\begin{proof} By proposition \ref{diffeo}, $\Phi:\left(\Sigma_1\backslash S_1\right)\longrightarrow\left(S_2\backslash S_1\right)$ is $2-1$; and since $\Phi:S_1\longrightarrow S_1$ maps $A\in S_1$ to $2A\in S_1$, $\Phi$ is clearly surjective. 
\end{proof}

\begin{theorem}\label{cover} $X_F$ is a double cover of $\Projan\mathcal{R}(N(X))$.
\end{theorem}
\begin{proof} Before we prove the general case let us consider the case where $F$ is the identity map, denoted by $Id$, $X=Id^{-1}(S_{n-2})=S_{n-2}$ and
$$
\begin{array}{rcl}
X_{Id}&=&\overline{\left\lbrace(h,l_1,l_2) | \rank h=n-2 \text{ and } (l_1,l_2)\in\mathbb{P}(\ker h)\right\rbrace}\\
\Projan\mathcal{R}(N(X))&=&C(S_{n-2}).
\end{array}
$$
Then, we are going to show that the map $\Psi:X_{Id}\longrightarrow C(S_{n-2})$ is a $2-1$ branched cover. From the group actions, it suffices to check that $\Psi$ is a $2-1$ branched cover on the fiber of $X_{Id}$ for the representatives
$$
h_r=\left(\begin{array}{c|c}
I_r & 0 \\
\hline
0&0
\end{array}\right)
$$ 
for all $r\in\{1,2,\ldots,n-2\}$. The fiber of $X_{Id}$ at each representative is equal to $\P^{n-r-1}\times\P^{n-r-1}$, since $\ker H_r=\C^{n-r}$. Now, for the fiber of $C(S_{n-2})$ at $h_r$ we are going to use the results from last section. If $r\neq n-2$, then, by theorem \ref{fiber}, the fiber of $C(S_{n-2})$ at $h_r$ is equal to $X_2(h_r)$, which means 
$$
C(S_{n-2})_{h_r}=\P(S_2)\subset\P(\Shr).
$$
If $r=n-2$, then $C(S_{n-2})_{h_{n-2}}=\P(S_2)=\P(\hom_s(\C^2,\C^2))$, since $h_{n-2}$ is a smooth point. Thus, for both cases we can consider the following map given by
$$
\begin{array}{lclc}
\varphi: & \P^{n-r-1}\times\P^{n-r-1} & \longrightarrow & \P(\Shr)\\
		& \left((S_i),(T_j)\right) & \longmapsto & \left(\begin{matrix}
												2S_1T_1 & \cdots & S_{n-r}T_1+T_{n-r}S_1\\
												\vdots & \ddots & \vdots\\												
												S_1T_{n-r}+T_1S_{n-r} & \cdots & 2S_{n-r}T_{n-r}
												\end{matrix}\right)
\end{array}
$$
This map $\varphi$ is the composition of the maps $\Phi\circ\sigma$, where $\sigma$ is the Segre embedding. Indeed, $Im(\sigma)=\P(\Sigma_1)$, which makes the composition well defined, and $\varphi(s,t)=\sigma(s,t)+\sigma(s,t)^t=\Phi(\sigma(s,t))=\P(S_2)$. Hence, $\Psi$ is a $2-1$ branched cover. The general case follows from the commutative diagram where the vertical arrows are embeddings on fibers.
$$
\xymatrix{
X_F  \ar[r] \ar[d] & \Projan\R(N(X)) \ar[d]\\
X_{Id} \ar[r]_{\Psi}       & C(S_{n-2}) }
$$
\end{proof}

As before, we use this theorem to compute the degree over the base $\mathbb{C}$ of the polar variety of dimension $1$ of $N(\mathcal{X})$, where $\mathcal{X}$ is the total space of the deformation, and a generic fiber is smooth. Consider the map
$$
\xymatrix{
 & \X\times\mathbb{P}^{n-1}\times\mathbb{P}^{n-1} \ar[dl]_{p_1} \ar[dr]^{p_2} & \\
\mathbb{P}^{n-1} & & \mathbb{P}^{n-1}}
$$
and let $h_1$ be the pullback of a hyperplane class of $\mathbb{P}^{n-1}$ via the projection map $p_1$, $h_2$ be the pullback of a hyperplane class of $\mathbb{P}^{n-1}$ via the projection map $p_2$ and $h$ be a hyperplane class on $\X\times\nP\times\nP$ defined as $h_1+h_2$. Denote the fiber over the origin in $\X$ of $\Projan\R(N(\X))$ by $E$, $p$ the projection to $\X$, and let $d$ be the dimension of $X$.

\begin{theorem}\label{C3C} Suppose $\X$ is a stabilization of $X$, with smooth base $\C$. The degree of the polar curve $\Gamma_d(N(\X))$ over $\C$ at the origin is
$$
\deg_{\C}\Gamma_d(N(\X))=\frac{1}{2}(h_1+h_2)^{d+2}\cdot\varphi^*E.
$$
\end{theorem}
\begin{proof} The reasoning is similiar to the proof of theorem \ref{HC}. In this case, the normal module $N(\X)$ has $\frac{n(n+1)}{2}$ generators and generic rank $3$, so $\Projan\R(N(\X))$ has dimension $d+3$, with generic fiber of dimension $2$. The fiber $E=p^{-1}(0)$ has dimension at most one less than $\Projan\R(N(\X))$, that is, $d+2$. The polar curve $\Gamma_d(N(\X))$ is given by intersecting $\Projan\R(N(\X))$ with $d+2$ generic hyperplanes of $\P^\frac{n(n+1)}{2}$ and projecting to $\X$ by $p$. Then,
$$
deg_{\C}\Gamma_d(N(\X))=[E]\cdot[h]^{d+2}.
$$
Now, consider the diagram 
$$
\xymatrix{
\X_F\cap\varphi^{-1}(h^{d+2}) \ar[r] \ar[dr] & \Projan\R(N(\X))\cap h^{d+2} \ar[d]\\
 & \C }
$$ 
By the theorem \ref{cover}, the map from $\X_F\cap\varphi^{-1}(h^{d+2})$ to $\Projan\R(N(\X))\cap h^{d+2}$ is $2-1$, which means that if $[E]\cdot [h]^{d+2}$ has $k$ points of intersection, then $[\varphi^{-1}(E)]\cdot[\varphi^{-1}(h^{d+2})]$ has $2k$ points. Thus,
$$
\begin{array}{ccl}
[\varphi^{-1}(E)]\cdot[\varphi^{-1}(h^{d+2})] & = &\varphi^*(E)\cdot\varphi^*(h^{d+2})\\
 & = & \varphi^*(E)\cdot(h_1+h_2)^{d+2}.
\end{array}
$$
Therefore, 
 $$
\deg_{\C}\Gamma_d(N(\X))=\frac{1}{2}(h_1+h_2)^{d+2}\cdot\varphi^*E.
$$
\end{proof}

Define $\Gamma_{i,j}(N(\X))$ to be the image of the projection of $\X_F\cap h_1^i\cap h_2^j$ over $\X$. We call these \textit{mixed polars} of type $(i,j)$ of $N(\X)$. By the same argument as in the proof of theorem \ref{C3C} the degree of the mixed polars is
$$
\deg_{\C}\Gamma_{i,j}(N(\X))=h_1^ih_2^j\cdot \varphi*E.
$$  
In order to simplify notation, let us write the degree of the mixed polars as $h_1^ih_2^j$. The degree of $\Gamma_d(N(\X))$ over $\C$ is
$$
\deg_{\C}\Gamma_d(N(\X))=\frac{1}{2}\sum_{i=0}^{d+2}\left(\begin{matrix}
																		d+2\\
																		i		
																		\end{matrix}\right)h_1^ih_2^{d+2-i}.
$$

The reasoning follows by intersection theory and the last theorem.

\section{Computing the Degrees of the Mixed Polars }

The calculation of the degrees of the mixed polars defined in the last section helps us to calculate the degree of the polar varieties $\Gamma_d(N(\X))$. In \cite{Gaff1}, Gaffney and Rangachev found an algorithm to compute these mixed polars when $\X$ is a maximum rank determinantal singularity, by using the colength of some specific ideals. We are not working with the maximal minors here, but we can still use some of the ideas presented in the paper to solve our problem. Our next step is to calculate the degree of the mixed polar $\Gamma_{i,j}(N(\X))$ for any possible dimension $d$ of $X$ and $\X$ a stabilization of $X$.

The mixed polar $\Gamma_{i,j}(N(\X))$, where $i\geq j$ and $i+j=d+2$, is defined by taking the hyperplane classes $\alpha^i=[S_n=\ldots=S_{n-i+1}=0]$ and $\beta^j=[T_1=\ldots=T_j=0]$ and their pullbacks via the projections $p_1$ and $p_2$, which are:
$$
\begin{array}{l}
h_1^i=\X\times (a_1:\ldots:a_{n-i}:0:\ldots:0)\times\P^{n-1}\\
h_2^j=\X\times\P^{n-1}\times (0:\ldots:0:b_{j+1}:\ldots:b_n).
\end{array}
$$
The mixed polar defined by projecting $\X_F\cap h_1^ih_2^j$ onto $\X$ is equal to:
$$
\Gamma_{i,j}(N(\X))=\overline{\left\lbrace x\in\X_{reg} \quad\vline\quad \exists\begin{array}{r}
(a_1,\ldots,a_{n-i},0,\ldots,0)\in\ker\widetilde{F}^t(x)\\
(0,\ldots,0,b_{j+1},\ldots,b_n)\in\ker\widetilde{F}(x)
\end{array}\right\rbrace}.
$$

\begin{theorem}\label{mixedi0} The degree of the mixed polar $\Gamma_{i,0}(N(\X))$ is equal to $0$.
\end{theorem} 
\begin{proof} By definition, the degree of $\Gamma_{i,0}(N(\X))$ is the intersection number $h_1^i\cdot\varphi^*E$. The dimension of $\P(\ker\widetilde{F}(x))$ is equal to $1$, which means that any generic fiber of $\X_F$ is isomorphic to $\P^1\times\P^1$. Take a point $(x,l_1,l_2)$ in $h_1^i\cap\varphi^*E$, because we do not have a $h_2$ term, $l_2$ can be any point in $\P^1$, that is, the intersection in this case cannot be a curve. Therefore, the degree of $\Gamma_{i,0}(N(\X))$ must be $0$.
\end{proof}

To calculate the degree of the defined mixed polar, we are going to consider sets in $\Sh$, prove results about them and then pull them back to $\X$ via $\widetilde{F}$. For this, consider the sets

$$
A(i,j,n)_l=\left\lbrace h\in\Sh\quad\vline\quad\begin{array}{rcl}
																\rank(p_{l-1}\circ\pi_{n-i}\circ h)&\leq&n-i-l\\
																\rank(p_{l}\circ h\circ\phi_{n-j})&\leq&n-j-1
																\end{array}\right\rbrace	
$$
where $\pi_{n-i}:\C^n\longrightarrow\C^{n-i}$ is the projection on the first $n-i$ factors, $p_l:C^k\longrightarrow\C^{k-l}$ is the projection on the last $k-l$ factors and $\phi_{n-j}:\C^{n-j}\longrightarrow\C^n$ is the canonical embedding of $\C^{n-j}$ into $\C^n$ on the last $n-j$ factors. 

The first inequality, called \textit{row condition}, analyzes the behavior of the submatrix of $h$, of size $(n-i-l+1)\times n$, that appears below:
$$
\includegraphics[scale=0.4]{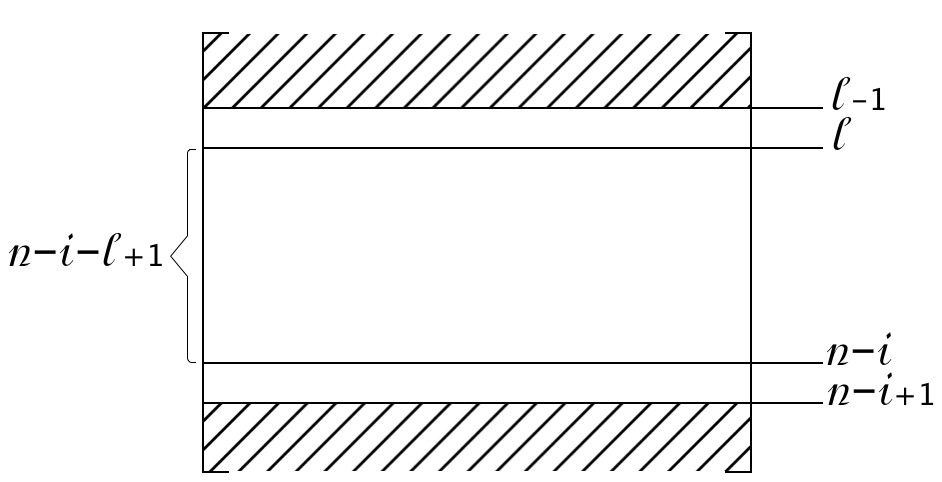}
$$
The row condition holds if this submatrix has rank less than the maximal rank. In order to have a well defined submatrix, that is, at least one row, $l$ must be less than or equal to $n-i$. Increasing $l$ by $1$ has the effect of dropping one more row on the upper side of the matrix while the bottom side remains the same, increasing the codimension by $1$. The second inequality, called \textit{column condition}, analyzes the behavoir of the submatrix of $h$, of size $(n-l)\times(n-j)$, that appears below:
$$
\includegraphics[scale=0.4]{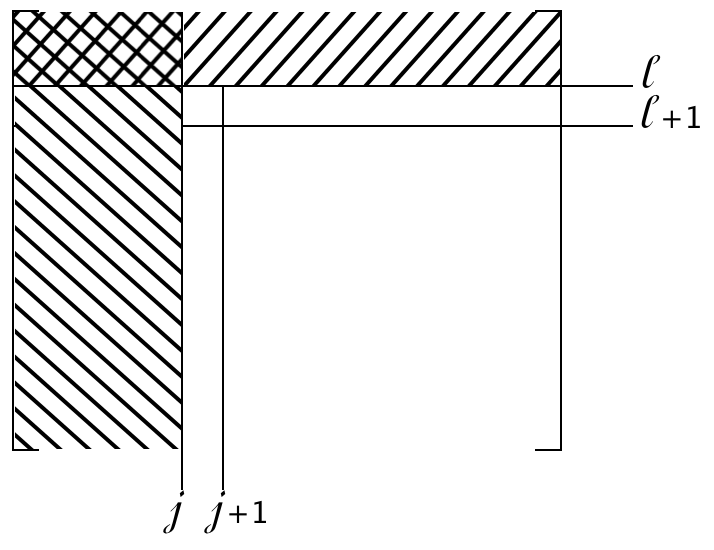}
$$
The column condition holds if the matrix has rank less than or equal to $n-j-1$. If $l=j+1$, then the column conditions holds for all $h\in\Sh$ since the resulting submatrix has exactly $n-j-1$ rows. For $l\in\{j+2,\ldots,n\}$, what happens is that $A(i,j,n)_l\subset A(i,j,n)_{j+1}$ because the row condition for $A(i,j,n)_l$ always implies the row condition for $A(i,j,n)_{l-1}$ and the column conditions for $A(i,j,n)_{j+2},\ldots,A(i,j,n)_n$ hold for all $h\in\Sh$. For this reason, we want $l$ to be less than or equal to $j+1$. Increasing $l$ by $1$ has the effect of dropping one more row on the upper side of the matrix while the bottom side remains the same, decreasing codimension by $1$. Note that the number of considered columns is the same as we change $l$. For now on we will always assume $j\leq i\leq n-1$ and $l\leq\min\{j+1,n-i\}$.

\begin{lemma} The row and column conditions for $A(i,j,n)_l$ define determinantal varieties.
\end{lemma}
\begin{proof} Let us start with the row condition. Since the row condition defines a variety by taking the vanishing of some minors of a matrix, the only thing needed to prove is that it has the right codimension. For this, we are going to use the same ideas A. Conca used in the section $2.3$ of \cite{SL}. Let $S^{n-i-l}(i,j,n)_ {l,R}$ be the variety defined by the row condition and $h$ be a matrix in it. Following Conca's notation, we have $M=0$ and $N=0$. Since rows $R_1,\ldots R_{l-1}$ can vary freely, we can move the block $S$ to obtain the format presented in Conca's paper. Then, we can take $Z$ as the rectangular $(n-i-l)\times n$ submatrix:
$$
\includegraphics[scale=0.5]{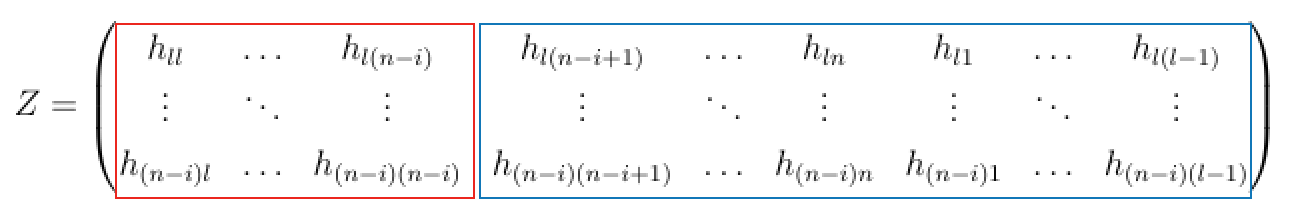}
$$
Then, $S$ is the square symmetric matrix marked as red (left block) and $P$ is the submatrix marked as blue (right block). By Conca's results $R_{n-i-l}(Z)$ is Cohen-Macaulay and also determinantal. Now, consider the projection map 
$$
\begin{array}{lclc}
\tau: & \Sh & \longrightarrow & \hom_{s,(n-i-l)}(\C^n,\C^{n-i-l})\\

		& h & \longmapsto & p_{l-1}\circ\pi_{n-i}\circ h
\end{array}
$$
where $\hom_{s,(n-i-l)}(\C^n,\C^{n-i-l})$ is the set of $(n-i-l)\times n$ matrices whose left upper $(n-i-l)\times(n-i-l)$ submatrix is symmetric. Since the preimage of $R_{n-i-l}(Z)$ is equal to $S^{n-i-l}(i,j,n)_ {l,R}$, we are done.

For the column condition, let $S^{n-j-1}(i,j,n)_ {l,R}$ be the variety defined by the row condition and $h$ be a matrix in it. Then, $Z$ is the rectangular $(n-l)\times(n-j)$ submatrix:
$$
\includegraphics[scale=0.5]{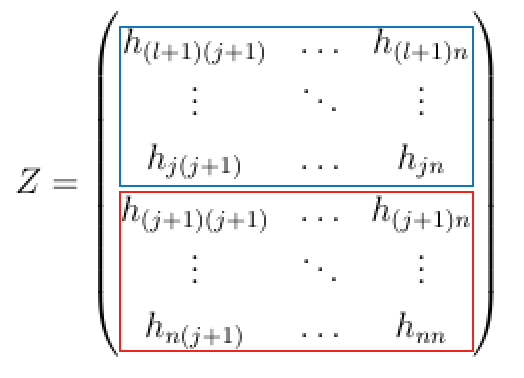}
$$
Following Conca's notation, we have $P=0$, $N=0$, $S$ is the square symmetric matrix marked as red (bottom block) and $M$ is the submatrix marked as blue (upper block). By Conca's results $R_{n-j-1}(Z)$ is Cohen-Macaulay and also determinantal. Now, consider the map 
$$
\begin{array}{lclc}
\tau: & \Sh & \longrightarrow & \hom_{s,(n-j-1)}(\C^{n-j},\C^{n-l})\\

		& h & \longmapsto & p_{l}\circ h\circ\phi_{n-j}
\end{array}
$$
where $\hom_{s,(n-j-1)}(\C^{n-j},\C^{n-l})$ is the set of $(n-l)\times(n-j)$ matrices whose bottom right $(n-j)\times(n-j)$ submatrix is symmetric. Since the preimage of $R_{n-j-1}(Z)$ is equal to $S^{n-j-1}(i,j,n)_ {l,C}$, we are done.
\end{proof}

\begin{proposition} $A(i,j,n)_l$ is a determinantal variety, for all $l$.
\end{proposition}
\begin{proof}
Let $S^{n-i-l}_{l,R}$ be the set satisfying the row condition for $A(i,j,n)_l$ and $S^{n-j-1}_{l,C}$ the set satisfying the column condition. Considering the column condition for 
$A(i,j,n)_l$ and Conca's notation, let $\hom_{s,(n-j)}(n-j,n-l)$ be the set of $(n-j)\times(n-l)$ matrices whose last $(n-j)\times(n-j)$ submatrix is symmetric. The proof will be by induction. Let us first consider the set $A(i,j,n)_{j+1}$. For this set, the column condition holds for all $h$, so this set is equal to its row condition, which is determinantal by the last lemma. Now, suppose that $A(i,j,n)_{l+1}$ is determinantal and consider $C$ a component of $A(i,j,n)_l$ which is not a component of $A(i,j,n)_{l+1}$. Take $h$ a generic element of $C$ such that $h$ does not satisfies the row condition for $A(i,j,n)_{l+1}$, ie., $\rank (R_{l+1},\ldots,R_{n-i})=n-i-l$, which also means that rows $R_{l+1},\ldots,R_{n-i}$ are linearly independent and $R_l$ depends on them. Take a small neighborhood of $h$, say $U$, that does not contain any elements of $A(i,j,n)_{l+1}$. The map given by
$$
\begin{array}{lclc}
F: & U & \longrightarrow & \hom_{s,(n-j-1)}(\C^{n-j},\C^{n-l})\\
					& a_{s,t} & \longmapsto & \{a_{s,t}\mid s\geq l+1,t\geq j+1\}
\end{array}
$$ 
is a submersion. Indeed, the linear independence of rows $R_{l+1},\ldots,R_{n-i}$ and the free variation of the last $i$ rows means that there is no relation on the coordinates of the target besides the symmetry ones. By abuse of notation let us consider that $S^{n-j-1}_{l,C}$ is inside of $\hom_{s,(n-j-1)}(\C^{n-j},\C^{n-l})$. Since $F$ is a submersion, the codimension of $F^{-1}(S^{n-j-1}_{l,C})$ in $S^{n-i-l}_{l,R}$ is the same as the codimension of $S^{n-j-1}_{l,C}$ in $\hom_{s,(n-j-1)}(\C^{n-j},\C^{n-l})$. Therefore, $A(i,j,n)_l$ is determinantal.
\end{proof}

Before we state the next proposition, let us define $\Gamma_{i,j}(S_{n-2})$ as the set given by taking $F$ as the identity map in the definition of the mixed polars, that is,
$$
\Gamma_{i,j}(S_{n-2})=\overline{\left\lbrace h\in S_{n-2}\backslash S_{n-3} \quad\vline\quad \exists\begin{array}{r}
(a_1,\ldots,a_{n-i},0,\ldots,0)\in\ker(h)\\
(0,\ldots,0,b_{j+1},\ldots,b_n)\in\ker(h^t)
\end{array}\right\rbrace}.
$$

The pullback of this set by $F$ is equal to $\Gamma_{i,j}(N(\X))$. For now we are going to use it to prove some results and, at the end, pull them back to finally calculate the degree of our mixed polar.

\begin{proposition}\label{GC1}$\Gamma_{i,j}(S_{n-2})\subset A(i,j,n)_1$.
\end{proposition}
\begin{proof}
Let $h$ be in $\Gamma_{i,j}(S_{n-2})$. Assume that our hyperplanes are chosen so that there are vectors $v=(v_1,\ldots,v_{n-1},0,\ldots,0)$, $u=(0,\ldots,0,u_{j+1},\ldots,u_n)$ with coordinates in $\C$ such that $v\cdot h=0$ and $h\cdot u=0$. The first equality says that the rows $R_1,\ldots,R_{n-i}$ are linearly dependent, which means that, $\rank(R_1,\ldots,R_{n-i})\leq n-i-1$. The second equality says that $C_{j+1},\ldots,C_n$ are linearly dependent, that is, $\rank(C_{j+1},\ldots,C_n)\leq n-j-1$. Therefore, $h\in A(i,j,n)_1$.
\end{proof}

Since $i+j+1=d+2+1=q$, these sets have the right codimension, meaning that their pullbacks will be curves on $\C^q\times\C$ and, therefore, we can calculate their degrees. The idea here is to construct a chain of sets starting with $\Gamma_{i,j}(S_{n-2})$, whose set of components is denoted by $C(i,j,n)_0$, and ending with a set which is determinantal. The degree of the mixed polar we want will be given by an alternating sum of the degrees of this sets, but in order to do that we need to understand how the chain ends. For this, we are going to analyze two cases: $n-i\leq j$ and $j<n-i$. Let us start proving the results we need for the case where $n-i\leq j$. 

\begin{proposition}\label{GC2} Suppose $n-i\leq j$. Then, $A(i,j,n)_{n-i}\subset A(i,j,n)_{n-i-1}$.
\end{proposition}
\begin{proof}
Let $h$ be a matrix in $A(i,j,n)_{n-i}$. The row condition for $A(i,j,n)_{n-i}$ implies that the row $R_{n-i}$ is zero, which means that $R_{n-i},R_{n-i-1}$ are linearly dependent. Hence $h$ satisfies the row condition for $A(i,j,n)_{n-i-1}$. On the other hand, the column condition for $A(i,j,n)_{n-i}$ says that the matrix given by the last $n-j$ columns and last $i$ rows of $h$ has less than maximal rank. Since $R_{n-i}$ is zero, the matrix given by the last $n-j$ columns and last $i+1$ rows of $h$ has also less than maximal rank, meaning that $h$ is in $A(i,j,n)_{n-i-1}$.
\end{proof}

\begin{lemma}\label{GP} Suppose $n-i\leq j$ and $C\in C(i,j,n)_l$ is a component of $A(i,j,n)_l$. If $h\in C$ is a generic point, then $\rank(p_{l}\circ h\circ \phi_{n-j})=n-j-1$ and $\rank(p_{l-1}\circ\pi_{n-i}\circ h)=n-i-l$.
\end{lemma}
\begin{proof}
Before we start the proof, recall that the points in $A(i,j,n)_l$ are defined by $\rank(p_{l-1}\circ\pi_{n-i}\circ h)$ and
$\rank(p_{l}\circ h\circ \phi_{n-j})$ both having less than maximal rank. Now, let $h$ be a generic point of $C$.

\begin{enumerate}
\item $\rank(p_{l}\circ h\circ \phi_{n-j})=n-j-1$.

Pick a column $C_p$ between columns $C_{j+1}\ldots C_n$ that is linearly dependent on the others and drop it. This gives us a submatrix of size $(n-l)\times(n-j-1)$. We would like to consider square submatrices, so let us drop rows $R_{l+1}\ldots R_{j+1}$. By varying the entries, as small as desired,  of the remaining square matrix we can have, as result, a matrix of rank exactly $n-j-1$. After making such perturbation, we change the entries of $C_p$ using its relation with the remaining columns and the variations on them; so then we will still have $C_p$ linearly dependent on $C_{j+1},\ldots,C_{p-1},C_{p+1},\ldots,C_n$. Since $n-i\leq j$, the row condition for $A(i,j,n)_l$ will not be affected. Therefore, the perturbation $\tilde{h}$ of $h$ lies in $A(i,j,n)_l$ and $\rank(p_{l}\circ\tilde{h}\circ \phi_{n-j})=n-j-1$. Since $h$ is a generic point, it satisfies this generic property. 

\item $\rank(p_{l-1}\circ\pi_{n-i}\circ h)=n-i-l$.

A similar argument will be made in here. Pick a row $R_p$ between rows $R_{l},\ldots, R_{n-i}$ that is linearly dependent on the others and drop it. This gives us a submatrix of size $(n-i-l)\times n$. We would like to consider square submatrices, so let us drop columns $C_1\ldots,C_l$. By varying the entries, as small as desired, of the remaining square matrix we can have, as result, a matrix of rank exactly $n-i-l$. After making such perturbation, we change the entries of $R_p$ using its relation with the remaining rows and the variations on them; so then we will still have $R_p$ linearly dependent on $R_l,\ldots,R_{p-1},R_{p+1},\ldots,R_{n-i}$. Since $n-i\leq j$, the column condition for $A(i,j,n)_l$ will not be affected. Therefore, the perturbation $\tilde{h}$ of $h$ lies in $A(i,j,n)_l$ and $\rank(p_{l-1}\circ\pi_{n-i}\circ\tilde{h})=n-i-l$. Since $h$ is a generic point, it satisfies this generic property. 
\end{enumerate}
\end{proof}

The next result is important for both cases. We will state it here and then use it again in the $j<n-i$ case.

\begin{proposition}\label{GC3} If $1\leq l<\min\{j+1,n-i\}$, then 
$$
C(i,j,n)_l\subset C(i,j,n)_{l-1}\cup C(i,j,n)_{l+1}.
$$
\end{proposition}   
\begin{proof} Let $h$ be an element in $A(i,j,n)_l$. To say that $h$ satisfies both row and column conditions is the same as saying that $p_{l-1}\circ\pi_{n-i}\circ h$ and $p_l\circ h\circ\phi_{n-j}$ are not submersions, and we are going to use this in our proof. Take $C$ an element of $C(i,j,n)_l$, that is, $C$ is one of the components of $A(i,j,n)_l$. Let $h$ be a generic point of $C$. For the row condition, $h$ has a submatrix of size $(n-i-l+1)\times n$ whose rank is equal to $n-i-l$. Consider the map $p_{l}\circ\pi_{n-i}\circ h:\C^n\longrightarrow\C^{n-i-l}$ used to define the row condition for $A(i,j,n)_{l+1}$. Our proof is divided into two situations: $p_{l}\circ\pi_{n-i}\circ h$ being a submersion or not. First, if the map is a submersion, then $\rank(p_{l}\circ\pi_{n-i}\circ h)=n-i-l$, meaning that $h$ is not in $A(i,j,n)_{l+1}$ and, therefore, $C\notin C(i,j,n)_{l+1}$. Now, we need to show that $h$ satisfies the row and column condition for $A(i,j,n)_{l-1}$. The map $p_{l}\circ\pi_{n-i}\circ h$ being a submersion means that the rows $R_{l+1},\ldots,R_{n-i}$ are linearly independent. Since, by the row condition of $A(i,j,n)_l$, $R_{l+1},\ldots,R_{n-i}$ are linearly dependent, we have that $R_l$ is dependent on $R_{l+1},\ldots,R_{n-i}$. This means that $h$ satisfies the row condition for $A(i,j,n)_{l-1}$. On the other hand, $R_l$ being dependent on $R_{l+1},\ldots,R_{n-i}$ means that $\rank(p_{l-1}\circ h\circ\phi_{n-j})=\rank(p_{l}\circ h\circ\phi_{n-j})\leq n-j-1$.Therefore, $h\in A(i,j,n)_{l-1}$ and there is a $Z$-open subset of $C$ contained in some component of $A(i,j,n)_{l-1}$, which means that $C\in C(i,j,n)_{l-1}$.

Now, if $p_{l}\circ\pi_{n-i}\circ h$ is not a submersion, then rows $R_{l+1},\ldots,R_{n-i}$ are linearly dependent, meaning that $h$ satisfies the row condition for $A(i,j,n)_{l+1}$. The column condition for $A(i,j,n)_l$ automatically implies the column condition for $A(i,j,n)_{l+1}$ since $\rank(p_{l+1}\circ h\circ\phi_{n-j})=\rank(p_{l}\circ h\circ\phi_{n-j})\leq n-j-1$. Therefore, $h\in A(i,j,n)_{l+1}$ and there is a $Z$-open subset of $C$ contained in some component of $A(i,j,n)_{l+1}$, which means that $C\in C(i,j,n)_{l+1}$.   
\end{proof}

\begin{proposition}\label{GC4} Suppose $n-i\leq j$ and $2\leq l\leq n-i$. Then,
$$
C(i,j,n)_l\cap C(i,j,n)_{l-2}=\emptyset.
$$ 
\end{proposition}
\begin{proof} The idea here is to take an element $h$ of some element $C$ of $C(i,j,n)_l$, then we vary the entries of $h$ continuously so we stay in $A(i,j,n)_l$ and the variation $\tilde{h}$ will not be in $A(i,j,n)_{l-2}$. We can also assume that $h$ is not in any other component of $A(i,j,n)_l$. Let $h$ be a generic point of $C$. Then, as seen before, $\rank(p_{l}\circ h\circ \phi_{n-j})=n-j-1$ and let $C_p$ be the column between $C_{j+1},\ldots,C_n$ which is linearly dependent on the others. Let $A$ be the square matrix given by the last $n-j-1$ rows and columns $C_{j+1}\ldots,C_{p-1},C_{p+1},\ldots,C_n$ of maximal rank, that is, its determinant is not zero. Now, consider the vector $v$ given by the cofactors $\left\lbrace c_{(j+1)(j+1)},\ldots,c_{(j+1)p},\ldots, c_{(j+1)n}\right\rbrace$ of row $R_{j+1}$. This means that $v\neq 0$ since $c_{(j+1)p}=\det A\neq 0$. Take, $\overline{v}$ the complex conjugate of $v$ and vary row $R_{l-1}$ by $t\overline{v}$. This variation does not affect the row and column conditions for $A(i,j,n)_{l}$, meaning that the resulting matrix $\tilde{h}$ is in $A(i,j,n)_{l}$ but it will not satisfy the column condition for $A(i,j,n)_{l-2}$. Indeed, $R_{l-1}$ is the first row considered in the column condition for $A(i,j,n)_{l-2}$. The variation $R_{l-1}+t\overline{v}$, for $t\neq 0$, gives us the following result when we consider all the rows starting at column $C_{j+1}$:
$$
\det\left[\begin{matrix}
R_{l-1}+t\overline{v}\\
R_{j+2}\\
\vdots\\
R_n
\end{matrix}\right]=
\det\left[\begin{matrix}
R_{l-1}\\
R_{j+2}\\
\vdots\\
R_n
\end{matrix}\right]+
t||v||^2\neq 0 . 
$$
Hence, the submatrix of $p_{l-2}\circ h\circ\phi_{n-j}$ consisting of rows $R_{l-1},R_{j+2},\ldots,R_n$ for slight variations has maximal rank, and, therefore, the column condition of $A(i,j,n)_{l-2}$ fails. Thus, $C\notin C(i,j,n)_{l-2}$. 
\end{proof}

In order to understand how the last results help us to calculate the degree of the mixed polar, let us look first at the sets $A(i,j,n)_1$, $A(i,j,n)_2$ and $A(i,j,n)_3$. Suppose $C_{21},\ldots,C_{2k}$ are the components of $A(i,j,n)_2$. By the proposition \ref{GC3}, some of these components are components of $A(i,j,n)_1$ and the others are components of $A(i,j,n)_3$. The proposition \ref{GC4} says that $A(i,j,n)_1$ and $A(i,j,n)_3$ have no components in common, therefore we can suppose that $C_{21},\ldots,C_{2p}$ are the components of $A(i,j,n)_2$ contained in $A(i,j,n)_1$ and that $C_{2(p+1)},\ldots,C_{2k}$ are the components contained in $A(i,j,n)_3$. In an alternating sum of degrees, all components of $A(i,j,n)_2$ would be canceled. Continuing this argument, we will have all the components being canceled, except for some components of $A(i,j,n)_1$ and some components of $A(i,j,n)_{n-i}$. The proposition \ref{GC2} takes care of the components of $A(i,j,n)_{n-i}$, since $A(i,j,n)_{n-i}$ is contained in $A(i,j,n)_{n-i-1}$. Therefore, the only components left are those in $A(i,j,n)_1$, which are all the components of $\Gamma_{i,j}(S_{n-2})$ by the proposition \ref{GC1} and the proposition \ref{GC4} for $l=2$.

Now, let us pull back these sets in order to obtain the degree of our mixed polars. The $A(i,j,n)_l$ are determinantal varieties, we can take hyperplanes generic enough so that $A_l(\widetilde{F})=\widetilde{F}^{-1}(A(i,j,n)_l)$ has the expected codimension $i+j+1=q$, which means we have curves on $\C^q\times\C$. Since determinantal varieties are Cohen Macaulay, the degrees of all $A_l(\widetilde{F})$ are calculated by taking the colength of the ideal given by the maximal minors. Therefore,
$$
\deg_{\C}(\Gamma_{i,j}(N(\X)))=\sum_{l=1}^{n-i}(-1)^{l+1}\deg_{\C}A_l(\widetilde{F}).
$$ 
The degrees of the sets $A_l(\widetilde{F})$ are based on the rows and columns of the map $\widetilde{F}$. However, we do not have our hands on such map, we only know the rows and columns of $F$ and, therefore, we would like to use them to find the degree of the polar variety $\Gamma_{i,j}(N(\X))$. For this, consider the projection $A_l(\widetilde{F})=A_l\longrightarrow\C$. The degree of $t$ is the colength of $(t)$ in $\O_{A_l}$ since $(A_l,0)$ is a Cohen-Macaulay variety. In this case we have
$$
\dim_{\C}\frac{\O_{A_l,0}}{(t)}=\dim_{\C}\frac{\O_{\C^q\times\C,(0,0)}}{(t,I_{A_l})}=\dim_{\C}\frac{\O_{q}}{I_{A_l}}.
$$
Therefore, for $n-i\leq j$, the degree over $\C$ of the mixed polars of $N(\X)$ is
$$
\deg_{\C}(\Gamma_{i,j}(N(\X)))=\sum_{l=1}^{n-i}(-1)^{l+1}\colength I_{A_l}.
$$

Now, let us procede to the results necessary to analyze the case where $j<n-i$. It is important to keep in mind that we will need to use some of the results previously stated.

\begin{proposition} Suppose $j<n-i$. Then, $A(i,j,n)_{j+1}\subset A(i,j,n)_l$.
\end{proposition}
\begin{proof} First of all, let us remember that $l\leq\min\{n-i,j+1\}$. Since, by hypothesis, $j<n-i\Rightarrow j+1\leq n-i$, we have $l\leq j+1$. In that case, let $h$ be a matrix in $A(i,j,n)_{j+1}$. By the row condition, the rows $R_{j+1},\ldots,R_{n-i}$ are linearly dependent, that is, $\rank(R_{j+1},\ldots,R_{n-i})\leq n-i-j-1$. By adding the rows $R_l,\ldots,R_j$, the rank will increase at most $j-l+1$, which means that for all $l<j+1$ $\rank(R_{j+1},\ldots,R_{n-i})\leq n-i-l$ making $h$ satisfy the row condition for all $A(i,j,n)_l$, $l<j+1$. Now, the fact that $R_{j+1},\ldots,R_{n-i}$ are linearly dependent implies that the columns $C_{j+1},\ldots,C_{n-i}$ are also linearly dependent, which means that $h\circ\Phi_{n-j}$ has less than maximal rank, satisfying the column condition for all $A(i,j,n)_l$, $l<j+1$. Therefore, $h\in A(i,j,n)_l$ for all $l<j+1$.
\end{proof}

\begin{proposition} Suppose $j<n-i$ and $3\leq l\leq j+1$. Then, 
$$
C(i,j,n)_l\cap C(i,j,n)_{l-2}=C(i,j,n)_{j+1}.
$$
\end{proposition}
\begin{proof}
The idea is to take an element $h$ of some component $C$ of $A(i,j,n)_l$ that it is not in any component of $A(i,j,n)_{j+1}$, then we vary entries of $h$ continuosly so we stay in $A(i,j,n)_l$ and the variation $\tilde{h}$ will not be in $A(i,j,n)_{l-2}$. Unfortunately, we cannot proceed as we did in the previous case because unlike the previous case, where there was no relation between column and row condition, here we may have a problem with a overlaping block caused by the fact that $j<n-i$. Let $h$ be a generic point of $C\in C(i,j,n,)_l$, where $C$ is not in $C(i,j,n)_{j+1}$. Since $p_j\circ\pi_{n-i}\circ h$ has maximal rank, let us freely deform $h$ by choosing a basis for the row space of $p_{l-1}\circ\pi_{n-i}\circ h$ by supplementing the rows of $p_j\circ\pi_{n-i}\circ h$, which are linearly independent, with as many additional rows as necessary. Then, in our deformation of rows $R_{j+1},\ldots,R_{n-i}$ we leave fixed the other rows in the basis, deforming the remaining rows using the relation relating each remaining row to the rows in the basis. What we are doing here is similar to the method of deformation seen in lemma \ref{GP}. This technique ensures that the rank of rows $R_{l},\ldots,R_{n-i}$ stay constant in small deformations, preserving the row condition for $A(i,j,n)_l$ and making the deformations stay inside $C$. 

The next step is to show that we can deform the lower right $(n-j)\times(n-j)$ block so that it has rank $n-j-1$, and $p_l\circ h\circ\phi_{n-j}$ has rank $n-j-1$ as well. Let $H$ be the kernel of the map defined by using the rows among $R_{l+1},\ldots,R_j$ of $h\circ\phi_{n-j}$ which are part of the basis elements for the image of $p_{l-1}\circ\pi_{n-i}\circ h$. Since $n-j>j-l$, $H$ is non-trivial. We can suppose that $\rank(p_j\circ h\circ\phi_{n-j})<n-j-1$. Then the kernel of $p_j\circ h\circ\phi_{n-j}$ must intersect $H$ non-trivially. Let $l$ be a line in this intersection. We can make small enough deformations of $p_j\circ h\circ\phi_{n-j}$ so that its kernel is exactly $l$. This includes a deformation of $p_l\circ h\circ\phi_{n-j}$ which has kernel rank $1$, and still satisfies the row condition. Since this deformation $\widetilde{h}$ is small, $\widetilde{h}$ is still in $C$ and $p_l\circ h\circ\phi_{n-j}$ has rank $n-j-1$.

Now we can deform row $R_{l-1}$ of $h\circ\phi_{n-j}$ without affecting the row condition for $A(i,j,n)_l$ so that the matrix of $p_{l-2}\circ\widetilde{h}\circ\phi_{n-j}$ has rank $n-j$ and, therefore, cannot be in $A(i,j,n)_{l-2}$.

\end{proof}
In order to state the final formula for the degree of the mixed polars when $j<n-i$, it is important to prove that $C(i,j,n)_0\cap C(i,j,n)_2=\emptyset$. For this, let $h$ be a generic element of a component $C$ of $A(i,j,n)_2$. This means that $\rank(p_{1}\circ\pi_{n-i}\circ h)=n-i-2$ and $\rank(p_{2}\circ h \Phi_{n-j})=n-j-1$. Without loss of generality we can consider $h$ as follows:
$$
h=\left(\begin{matrix}
1&0&0&0&\ldots&0&0\\
0&0&0&0&\ldots&0&0\\
0&0&1&0&\ldots&0&0\\
0&0&0&1&\ldots&0&0\\
\vdots&\vdots&\vdots&\vdots&\ddots&\vdots&\vdots\\
0&0&0&0&\cdots&1&0\\
0&0&0&0&\cdots&0&0
\end{matrix}\right)
$$
since the rank of the following $(n-i-1)\times n$ submatrix is:  
$$
\rank\left(\begin{matrix}
0&0&0&0&\ldots&0&0&\ldots&0\\
0&0&1&0&\ldots&0&0&\ldots&0\\
0&0&0&1&\ldots&0&0&\ldots&0\\
\vdots&\vdots&\vdots&\vdots&\ddots&\vdots&\vdots&\ddots&\vdots\\
0&0&0&0&\cdots&1&0&\ldots&0\\
\end{matrix}\right)= n-i-2
$$
and the rank of the following $(n-2)\times(n-j)$ submatrix is:  
$$
\rank\left(\begin{matrix}
1&0&\ldots&0&0\\
0&1&\ldots&0&0\\
\vdots&\vdots&\ddots&\vdots&\vdots\\
0&0&\ldots&1&0\\
0&0&\cdots&0&0\\
\end{matrix}\right)= n-j-1.
$$
The path $h+t(\delta_{1n}+\delta_{n1})$ satisfies the row condition for $A(i,j,n)_2$ because this condition does not include the first and last rows, and also satisfies the column condition for $A(i,j,n)_2$ because this condition does not include the first and last column either. However, the determinant of the $(n-1)\times(n-1)$ submatrix formed by dropping the second column and second row is easily seen as non-zero for $t\neq 0$. So, the generic point of the path does not lie in $S_{n-2}$, although it is in $A(i,j,n)_2$. Therefore, the component $C$ of $A(i,j,n)_2$ is not contained in any component of $\Gamma_{i,j}(S_{n-2})$.

Finally, proceeding as in the previous case we still have an alternating sum of degrees, and some of the terms will be canceled because of proposition \ref{GC3}. However, in this case the intersection is not always empty, so we need to be careful about it. Inside of each $C(i,j,n)_l$, $l\neq 0$ we have a copy of $C(i,j,n)_{j+1}$. If $j$ is odd, then all the copies will cancel each other and, again, we just need to worry about the remaining components of $A(i,j,n)_1$. But since $A(i,j,n)_2$ has no components in common with $A(i,j,n)_0$, the only components left in the alternating sum are the components of the mixed polar in question, wich means that,
$$
\deg_{\C}(\Gamma_{i,j}(N(\X)))=\sum_{l=1}^{j+1}(-1)^{l+1}\colength I_{A_l}.
$$ 
Now, if $j$ is even, then we need to take an extra term $\deg_{\C}A_{j+1}(\widetilde{F})$ into consideration. Therefore,
$$
\deg_{\C}(\Gamma_{i,j}(N(\X)))=\left(\sum_{l=1}^{j+1}(-1)^{l+1}\colength I_{A_l}\right)-\colength I_{A_{j+1}}.
$$ 

\begin{example} Consider the map
$$
\begin{array}{lclc}
F: & \mathbb{C}^5 & \longrightarrow & \hom_s(\C^4,\C^4)\\
		& (x_1,x_2,x_3,x_4,x_5) & \longmapsto & \left(\begin{array}{cccc}
												x_1&x_2&x_3&x_4\\
												x_2&x_3&x_4&x_5\\
												x_3&x_4&x_5&x_1\\
												x_4&x_5&x_1&2x_2
												\end{array}\right)
\end{array}
$$
and the symmetric determinantal variety $X=F^{-1}(S_2)$ of dimension $2$ with smoothing $\X$. First, let us calculate $\Gamma_{3,1}(N(\X))$. The set $A(3,1,n)_1$ is given by

$$
A_1(\widetilde{F})=\left\lbrace x\in\C^5\times\C \quad\vline\quad\begin{array}{r}
														\rank\left(\begin{matrix}
														x_1 & x_2 & x_3&x_4
														\end{matrix}\right)\leq 0\\
														\\
														\rank\left(\begin{matrix}
														x_3&x_4&x_5\\
														x_4&x_5&x_1\\
														x_5&x_1&2x_2
														\end{matrix}\right)\leq 2
														\end{array}\right\rbrace.
$$
The degree of $A_1(\widetilde{F})$ is the colength of
$$
I_{A_1}=(x_1,x_2,x_3,x_4,2x_2x_3x_5-x_3x_1^2+2x_1x_4x_5-2x_2x_4^2-x_5^3)
$$
which can be easily calculated by hand as $\colength I_{A_1}=3$. As in the curve case, $A_2(\widetilde{F})$ is empty. In fact, $A_1(\widetilde{F})$ is equal to $\Gamma_{3,1}(N(\X))$. Now, let us calculate $\Gamma_{2,2}(N(\X))$. The set $A_1(\widetilde{F})$ here is given by

$$
\begin{array}{l}
A_1(\widetilde{F})=\left\lbrace x\in\C^5\times\C \quad\vline\quad\begin{array}{r}
														\rank\left(\begin{matrix}
														x_1&x_2&x_3&x_4\\
														x_2&x_3&x_4&x_5
														\end{matrix}\right)\leq 1\\
														\\
														\rank\left(\begin{matrix}
														x_4&x_5\\
														x_5&x_1\\
														x_1&2x_2
														\end{matrix}\right)\leq 2
														\end{array}\right\rbrace
\end{array}.
$$
The degree of $A_1(\widetilde{F})$ is the colength of
$$
I_{A_1}=\left(\begin{array}{c}
x_1x_3-x_2^2,x_1x_4-x_2x_3,x_1x_5-x_2x_4\\
x_2x_4-x_3^2,x_2x_5-x_3x_4,x_3x_5-x_4^2\\
x_4x_1-x_5^2,2x_4x_2-x_5x_1,2x_5x_2-x_1^2
\end{array}\right).
$$
According to \textit{Singular}, $\colength I_{A_1} =12$. 

The set $A_2(\widetilde{F})$ is equal to
$$
\begin{array}{l}
A_2(\widetilde{F})=\left\lbrace x\in\C^5\times\C \quad\vline\quad\begin{array}{r}
														\rank\left(\begin{matrix}
														x_2&x_3&x_4&x_5
														\end{matrix}\right)\leq 0\\
						                                    \\
														\rank\left(\begin{matrix}
														x_5&x_1\\
														x_1&2x_2
														\end{matrix}\right)\leq 2
														\end{array}\right\rbrace
\end{array}.
$$
The degree of $A_2(\widetilde{F})$ is the colength of
$$
I_{A_2}=(x_2,x_3,x_4,x_5,2x_5x_2-x_1^2)
$$
which can be easily calculated by hand as $\colength I_{A_2}=2$. Thus,
$$
\deg_{\C}\Gamma_{2,2}(N(\X))=12-2=10.
$$
Therefore,
$$
\begin{array}{ccl}
\deg_{\C}\Gamma_2(N(\X))&=&4\deg_{\C}\Gamma_{3,1}(N(\X))+3\deg_{\C}\Gamma_{2,2}(N(\X))\\
 &=&4\cdot 3+3\cdot 10\\
 &=&42. 
\end{array}
$$
\end{example}

For more examples see \cite{SDSWE}.

 \end{document}